\documentclass{article}
\usepackage[utf8]{inputenc}
\usepackage[margin=1in]{geometry}
\usepackage{amsmath} 
\usepackage{amssymb} 
\usepackage{mathtools}
\usepackage{acronym}
\usepackage{mathrsfs}

\usepackage{subcaption}
\usepackage{hyperref}
\hypersetup{colorlinks=true, linkcolor=black}
\usepackage{algorithm}%
\usepackage{algorithmicx}%
\usepackage{algpseudocode}%

\newcommand{\Rn}{\mathbb{R}}

\newcommand{\Nn}{\mathbb{N}}
\newcommand{\Lie}{\mathcal{L}} 
\newcommand{\scL}{\mathscr{L}}

\newcommand{\M}{\mathbb{M}}

\newcommand{\gs}{\mathcal{G}}
\newcommand{\es}{\mathcal{E}}
\newcommand{\hs}{\mathcal{H}}
\newcommand{\vs}{\mathcal{V}}

\usepackage{mathtools}
\DeclarePairedDelimiter{\abs}{\lvert}{\rvert}
\DeclarePairedDelimiter{\norm}{\lVert}{\rVert}
\DeclarePairedDelimiter{\ceil}{\lceil}{\rceil}

\DeclarePairedDelimiterX{\inp}[2]{\langle}{\rangle}{#1, #2}

\DeclarePairedDelimiter{\Mp}{\mathcal{M}_+(}{)}

\usepackage{amsthm}
\newtheorem{theorem}{Theorem}
\newtheorem{assumption}[theorem]{Assumption}%
\newtheorem{remark}{Remark}

\title{\LARGE \bf
Peak Estimation of Hybrid Systems 
\\ with Convex Optimization
}
\renewcommand\footnotemark{}

\author{Jared Miller$^1$, Mario Sznaier$^1$
\thanks{$^1$J. Miller and M. Sznaier are with the Robust Systems Lab,  ECE Department, Northeastern University, Boston, MA 02115. (e-mails: miller.jare@northeastern.edu,  msznaier@coe.neu.edu).}
\thanks{ J. Miller and M. Sznaier were partially supported by NSF grants  CNS--1646121, CMMI--1638234, ECCS--1808381 and CNS--2038493,  and AFOSR grant FA9550-19-1-0005. 
This material is based upon research supported by the Chateaubriand Fellowship of the Office for Science \& Technology of the Embassy of France in the United States.
}
}
\date{\empty}
\begin{document}

\maketitle

\abstract{Peak estimation of hybrid systems aims to upper bound extreme values of a state function along trajectories, where this state function could be different in each subsystem. 
This finite-dimensional but nonconvex problem may be lifted into an  infinite-dimensional linear program (LP) in occupation measures with an equal objective under mild finiteness/compactness and smoothness assumptions. This LP may in turn be approximated by a convergent sequence of upper bounds attained from solutions of Linear Matrix Inequalities (LMIs) using the Moment-Sum-of-Squares hierarchy. The peak estimation problem is extended to problems with uncertainty and safety settings, such as measuring the distance of closest approach between points along hybrid system trajectories and unsafe sets.
    
    
    
}


\algblock{Input}{EndInput}
\algnotext{EndInput}
\algblock{Output}{EndOutput}
\algnotext{EndOutput}
\newcommand{\Desc}[2]{\State \makebox[2em][l]{#1}#2}


\maketitle

\section{Introduction}
\label{sec:intro}

This paper interprets and extends the peak estimation problem to dynamical systems with hybrid behavior. Peak estimation is the analysis problem of finding extremal values of state functions along system trajectories, such as finding the maximum speed of a craft given a set of initial conditions. 
A hybrid system is a dynamical system that possesses both continuous-time and discrete-time  dynamics \cite{goebel2009hybrid}. Hybrid systems have a wide array of applications, including walking robots \cite{collins2005efficient}, power converters \cite{zaupa2021resonant}, sampled-data control \cite{naghshtabrizi2006sampled}, and systems biology \cite{fischer2012redbloodcell}.
In this work (extending methods from \cite{zhao2019optimal}), the hybrid system is defined with respect to a series of spaces known as `locations' in which the hybrid trajectory evolves according to per-location \ac{ODE} dynamics. When the hybrid trajectory encounters a guard surface, it will transition to a (possibly) new location according to a reset map and continue its \ac{ODE} evolution. Peak estimation of hybrid systems equips each location with a state function, and the output of the peak estimation problem is the maximum state function value obtained across all locations by all hybrid systems trajectories starting from a set of initial conditions in a given time horizon. 



The \ac{ODE} peak estimation problem is an instance of an input-less \ac{OCP} with a free terminal time and zero running cost.
In the \ac{ODE} case, such \acp{OCP} are finite-dimensional but generally nonconvex to solve. This difficulty is exacerbated by the addition of hybrid dynamics. In both cases for a maximization objective, lower bounds on the true peak cost can be computed by approximate sampling, while upper bounds must be satisfied for all admissible trajectories. 
The foundational work in \cite{lewis1980relaxation} represented \ac{ODE} \acp{OCP} as infinite-dimensional \acp{LP} in nonnegative Borel Measures, and gave necessary conditions under which the \ac{OCP} and its \ac{LP} outer-approximation have the same optimal value (no relaxation gap). The \acp{LP} in \cite{lewis1980relaxation} involve terminal measures and occupation measures, which describe all information needed to reconstruct families of trajectories \ac{ODE}.
The work in \cite{cho2002linear} treated the peak estimation problem as an infinite-dimensional \ac{LP} in measures, and proposed approximations solved based on successively refined finite-dimensional gridded \acp{LP}. In the peak estimation setting, the necessary conditions for no relaxation gap are mild compactness and Lipschitz regularity requirements. Another approach to solving infinite-dimensional \acp{LP} is through the Moment-\ac{SOS} hierarchy, which truncates the \acp{LP} into \acp{SDP} of increasing complexity to produce tightening outer approximations  \cite{lasserre2009moments, tacchi2022convergence}. The \ac{SOS} side of this hierarchy was used in \cite{fantuzzi2020bounding} to solve peak estimation problems. Other applications of the Moment-\ac{SOS} hierarchy in analysis and control of \acp{ODE} includes \acp{OCP} \cite{lasserre2008nonlinear}, reachable set estimation \cite{henrion2013convex}, and maximum control invariant set estimation \cite{korda2014convex}.

Measures and the Moment-\ac{SOS} hierarchy have also been applied to solve problems featuring hybrid dynamical systems. Instances of these extensions include \acp{OCP} \cite{zhao2019optimal, zhao2019switching, rocca2018bio} and reachable sets \cite{shia2014reachable, mohan2016convex}.
Barrier functions to certify safety of hybrid system trajectories with respect to unsafe sets may also be found by \ac{SOS} programming \cite{prajna2004safety}.



The contributions of this paper are as follows:
\begin{itemize}
    \item Application of measure techniques to peak estimation for hybrid systems
    \item Use of Zeno caps to prevent unbounded executions
    \item A modular MATLAB framework for posing peak estimation problems for hybrid systems
    \item Extensions of existing work on peak estimation for uncertainty and safety analysis to hybrid systems
\end{itemize}

The paper is organized as follows. Section \ref{sec:prelim} introduces the notation convention and reviews background information. Section \ref{sec:hybrid_peak} formulates an infinite-dimensional measure program for peak estimation of hybrid system and its associated \ac{LMI} relaxation. Section \ref{sec:extensions} extends the hybrid peak estimation framework to safety analysis and possibly uncertain dynamical systems. Numerical examples are presented in Section \ref{sec:examples}. The paper is concluded in Section \ref{sec:conclusion}.

\section{Preliminaries}
\label{sec:prelim}

\begin{acronym}[WSOS]



\acro{LMI}{Linear Matrix Inequality}
\acroplural{LMI}[LMIs]{Linear Matrix Inequalities}
\acroindefinite{LMI}{an}{a}


\acro{LP}{Linear Program}
\acroindefinite{LP}{an}{a}

\acro{OCP}{Optimal Control Problem}
\acroindefinite{OCP}{an}{a}

\acro{ODE}{Ordinary Differential Equation}
\acroindefinite{ODE}{an}{a}


\acro{PSD}{Positive Semidefinite}


\acro{SDP}{Semidefinite Program}
\acroindefinite{SDP}{an}{a}

\acro{SOS}{Sum of Squares}
\acroindefinite{SOS}{an}{a}


\end{acronym}

\subsection{Notation}
\label{sec:notation}


The set of real numbers is $\Rn$ and of natural numbers is $\Nn$. The set of polynomials with real coefficients in indeterminates $x$ is $\Rn[x]$. Every polynomial $g \in \Rn[x]$ may be uniquely expressed as $g = \sum_{\alpha \in \Nn^n} g_\alpha x^\alpha$ in multi-index notation $x^\alpha = x_1^{\alpha_1} x_2^{\alpha_2}\ldots $ for some finite number of nonzero coefficients $g_\alpha$. The degree of a monomial $x^\alpha$ is $\abs{\alpha} = \sum_i \alpha_i$, and the degree of a polynomial $g$ is the maximum such $\abs{\alpha}$ where $g_\alpha \neq 0$. 

A nonnegative Borel measure supported in a set $X$ is a function that assigns each element of the $\sigma$-algebra of sets over $X$ with a nonnegative number (the `size' or `measure' of the set). The measure $\mu$ follows the rules $\mu(\varnothing) = 0$ and $\mu(A \cup B) = \mu(A) + \mu(B)$ if $A \cap B = \varnothing$ \cite{tao2011introduction}. The support of a nonnegative Borel measure is the locus of points $x$ where every open neighborhood $N(x)$ has $\mu(N(x)) > 0$. The set of all nonnegative Borel measures supported in $X$ is $\Mp{X}$. The space of continuous functions is $C(X)$, and a pairing between a nonnegative measure $\mu \in \Mp{X}$ and a function $f \in C(X)$ may be defined by Lebesgue integration $\inp{f}{\mu} = \int f d \mu = \int_X f(x) d \mu(x)$. The set $C^1(X) \subset C(X)$ is the set of continuous functions with continuous first derivatives. 
The indicator function $I_A$ of a set $A \subseteq X$ takes on the value $I_A(x) = 1$ if $x \in A$ and $I_A(x)=0$ otherwise, and satisfies the rule $\inp{I_A}{\mu} = \mu(A)$ for all $A \subseteq X$.
The mass of a measure $\mu$ is $\mu(X) = \inp{1}{\mu}$, and $\mu$ is a probability measure if $\inp{1}{\mu} = 1$. 
The Dirac delta $\delta_{x = x'} \in \Mp{X}$ is a probability measure supported only at the point $x=x'$, which follows the pairing rule $\inp{f}{\delta_{x=x'}} = f(x')$ for all test functions $f \in C(X)$. A nonnegative Borel measure supported at $r$ distinct points (atoms) is termed a rank-$r$ atomic measure. Such an atomic measure may be formed by a conic combination of Dirac deltas.
The projection  $\pi^x: X \times Y \rightarrow X$ is the map $(x,y) \mapsto x$. Given a mapping $Q: X \rightarrow Y$ and a measure $\mu(x) \in \Mp{X}$, the pushforward $Q_\# \mu(y)$ is the unique measure satisfying $\forall f \in C(Y): \inp{f(Q(x))}{\mu(x)} = \inp{f(y)}{Q_\# \mu(y)}.$

\subsection{Hybrid Systems}
\label{sec:hybrid}

The hybrid systems in this paper are posed over a set of $L$ locations. Each location $\ell = 1..L$ has state variables $x_\ell$ contained in the space $X^\ell \subseteq \Rn^{n_\ell}$. The subsystems obey nominal locally Lipschitz dynamics $f_\ell$ that satisfy
\begin{align}
    \label{eq:dynamics_nominal}\dot{x}_\ell(t) &= f_\ell(t, x_\ell(t)) & & \forall \ell = 1..L.
\end{align}

Available transitions between subsystems may be represented by a directed multigraph. A multigraph is a graph where pairs of vertices may be connected by multiple distinct edges \cite{hartsfield2013pearls}. Let $\gs = (\vs, \es)$ be a multigraph where each of the $L$ vertices of $\vs$ corresponds to a location. 
Each edge $e \in \es \subset \vs \times \vs$ 
is a directed arc from a source $\textrm{src}(e)$ to a destination $\textrm{dst}(e)$. Self-loops with $\textrm{src}(e) = \textrm{dst}(e)$ are permitted in this class of multigraphs.
Edges $e$ are associated with a guard $S_e$ and a reset map $R_e$. The guard $S_e$ is a subset of $X_{\textrm{src}(e)}$, and the reset map $R_e: X_{\textrm{src}(e)} \rightarrow X_{\textrm{dst}(e)}$ effects the transition.
The hybrid system is fully encoded by the tuple $\hs = (X, f, \gs, S, R)$ with attributes:
\begin{align}
    X &= \{X_\ell\}_{\ell=1}^L & & \textrm{State Spaces}\nonumber \\
    f &=\{f_\ell\}_{\ell=1}^L& &  \textrm{Dynamics}\nonumber \\
    \gs &= (\vs, \es)& & \textrm{Transition Multigraph}\nonumber \\
    S &= \{S_e\}_{e \in \es}& & \textrm{Guard Surfaces}\nonumber \\
    R &= \{R_e\}_{e \in \es} & & \textrm{Reset Maps}\nonumber
\end{align}
Execution of a hybrid system with multigraph transitions is based on Algorithm 1 of \cite{shia2014reachable}. An additional input is a set of Zeno caps $\{N_e\}_{e \in \es}$ which halt trajectory execution if any transition $e$ is traversed at least $N_e$ times \cite{ames2006there}.
The output of the following Algorithm \ref{alg:hybrid_exec} is a system trajectory $x(t)$, as well as records $\mathcal{T}, \mathcal{C}$ containing information about the times and locations of state transitions respectively. 
\begin{algorithm}
 \caption{\label{alg:hybrid_exec}Execution of Hybrid System $\hs$}
        \begin{algorithmic}
  \Input
  \Desc{$x_0$}{Initial Point}
  \Desc{$\ell_0$}{Initial Location}
  \Desc{$\hs$}{Hybrid System}
  \Desc{$T$}{Maximal Time}
  \Desc{$N$}{Transition Caps}
  \EndInput
  \Output
  \Desc{$x(t)$}{Trajectory of System}
  \Desc{$\mathcal{T}$}{Time Breaks}
  \Desc{$\mathcal{C}$}{Location Breaks}
  \Desc{$\mathcal{N}$}{Transition Counts}
  \EndOutput
\State Initialize Trajectory    $t \leftarrow 0, \ \ell \leftarrow \ell_0, \ x(0) \leftarrow x_0$
\State Initialize Traces $\mathcal{T} \leftarrow \{0\}, \mathcal{C} \leftarrow \{\ell\}, \mathcal{N} \leftarrow \{0\}_{e \in \es}$
\Loop
\State Follow dynamics $x'(s) = f_\ell(t, x(s))$ until $x(t)$ reaches a guard or $t=T$.
\If{$t = T$ \textbf{OR} $\not\exists S_e: \ x(t) \in S_e$ and $\textrm{src}(e) = \ell$, \textbf{OR}  $\exists e: \ \mathcal{N}_e = N_e$} 
        \State \textbf{halt}
            \EndIf
\State Find a guard $S_e$ with $x(t) \in S_e$ and $\textrm{src}(e) = \ell$ 
\State Append $t$ to $\mathcal{T}$ and $\textrm{dst}(e)$ to $\mathcal{C}$
\State Increment $\mathcal{N}_e \leftarrow \mathcal{N}_e + 1$ 
\State  Transition to $\ell \leftarrow \textrm{dst}(e), \  x(t) \leftarrow R_e(x(t))$
\EndLoop
\end{algorithmic}
\end{algorithm}
        
The trajectory $x(t)$ is well-defined when the time horizon $T$ and Zeno caps $N_e$ for all $e \in \es$ are finite and the guard surfaces $S_e$ are codimension-1. The trajectory $x(t)$ induces a function $\textrm{Loc}: [0, T] \rightarrow 1..L$ which returns the residing location of $x(t)$ at time $t$. Execution requires the following assumption of transversality,
\begin{assumption}
Let $x_\ell(t)$ be a segment of this trajectory that emerged from a transition $(\ell', \ell)$ at time $t^-$. For all guards $S^e$ with $\textrm{src}(e) = \ell$ such that $x_\ell(t) \in S^e$, the dynamics vector $f(t, x_\ell(t))$ possesses a normal component with respect to the tangent space of $S^e$ at $x_\ell(t)$. This implies that the time elapsed between any two resets is bounded below by some $\delta > 0$. 
\end{assumption}

\subsection{Peak Estimation and Occupation Measures}
\label{sec:peak}

The \ac{ODE} (non-hybrid) peak estimation setting involves a trajectory $x(t \mid x_0)$, starting at the initial point $x_0 \in X_0 \subset X$, and evolving according to dynamics $\dot{x}(t) = f(t, x(t))$ in a space $X$. The program to find the maximum value of a state function $p(x)$ along trajectories is
\begin{equation}
    \begin{aligned}
    P^* = & \sup_{t,\, x_0 \in X_0} p(x(t \mid x_0)), &\label{eq:peak_traj} & & 
    & \dot{x}(t) = f(t, x(t)).
    \end{aligned}
\end{equation}


The extremum $P^*$ may be bounded through the use of occupation measure relaxations \cite{cho2002linear}. 
An optimal trajectory satisfying $P^* = p(x^*) = p(x(t^* \mid x_0^*))$ is described by a triple $( x_0^*, t^*, x^*)$ \cite{miller2020recovery}.
The initial probability measure $\mu_0 \in \Mp{X_0}$ is distributed over the set of initial conditions $x_0 \sim \mu_0$. The peak measure $\mu_p \in \Mp{[0, T] \times X}$ is a terminal measure with free terminal time. 
For a stopping time $t^*$ and subsets $A \subseteq [0, t^*], \ B \subseteq X$, the occupation measure $\mu \in \Mp{[0, T] \times X}$ has a definition \cite{cho2002linear}
\begin{equation*}
    \mu(A \times B) =  \int_{X_0}\int_{t=0}^{t^*} I((t, x(t \mid x_0)) \in A \times B) dt \  d\mu_0(x_0).
\end{equation*}

The Lie derivative operator $\Lie_f$ may be defined for all test functions $v \in C^1([0, T]\times X)$
\begin{equation}
    \Lie_f v(t, x) = \partial_t v(t,x) + f(t,x) \cdot \nabla_x v(t,x).
\end{equation}
The three measures ($\mu_0, \ \mu_p, \ \mu$) are linked by Liouville's equation for all test functions $v$ 
\begin{align}
\inp{v(t, x)}{\mu_p} &= \inp{v(0, x)}{\mu_0} + \inp{\Lie_f v(t,x)}{\mu} \label{eq:liou_int}\\
\mu_p &= \delta_0 \otimes \mu_0 + \Lie_f^\dagger \mu \label{eq:liou_mom}
\end{align}

Liouville's equation ensures that initial conditions distributed as $\mu_0$ are connected to terminal points distributed as $\mu_p$ by trajectories following dynamics $f$.
A convex measure relaxation of problem \eqref{eq:peak_traj} is \cite{cho2002linear},
\begin{subequations}
\label{eq:peak_meas}
\begin{align}
p^* = & \ \sup \quad \inp{p(x)}{\mu_p} \label{eq:peak_meas_obj} \\
    & \mu_p = \delta_0 \otimes\mu_0 + \Lie_f^\dagger \mu \label{eq:peak_meas_flow}\\
    & \inp{1}{\mu_0} = 1 \label{eq:peak_meas_prob}\\
    & \mu, \mu_p \in \Mp{[0, T] \times X} \label{eq:peak_meas_peak}\\
    & \mu_0 \in \Mp{X_0}. \label{eq:peak_meas_init}
\end{align}
\end{subequations}
Constraint \eqref{eq:peak_meas_prob} ensures that both $\mu_0$ and $\mu_p$ are probability measures with unit mass. The objective \eqref{eq:peak_meas_obj} is the expectation of $p(x)$ with respect to $\mu_p$. 
There will be no relaxation gap between problems \eqref{eq:peak_traj} and  \eqref{eq:peak_meas_prob} ($P^* = p^*$) when $[0, T] \times X$ is compact and $f$ is Lipschitz \cite{lewis1980relaxation, cho2002linear, fantuzzi2020bounding}. Program \eqref{eq:peak_meas_init} is a particular form of the optimal control program from \cite{lewis1980relaxation} with zero running cost and  free terminal time.





\section{Peak estimation of hybrid systems}
\label{sec:hybrid_peak}

\subsection {Peak Program}
Let $X_0 = \{X_{0 \ell}\}$ be the set of initial conditions for system trajectories. Each of these system trajectories lie inside the set $X = \{X_\ell\}$.

Each location $\ell$ has a state cost $p_\ell: X_\ell \rightarrow \Rn$ and a set of initial conditions $X_{0 \ell} \subset X_\ell$. Each $p_\ell$ is either bounded below or constant at $-\infty$, and at least one $p_\ell$ is bounded.
The goal of peak estimation is to find the trajectory $x(t)$ which maximizes the state cost across all trajectories and locations: 
    \begin{align}
    P^* = & \sup_{t, \;  \ell_0 \; x_0 } \max_\ell p_\ell(x(t \mid x_0)) \qquad  x(t) \in X_\ell \nonumber\\
    & \textrm{Dynamics follow Algorithm \ref{alg:hybrid_exec} with input } (\ell_0, x_0, \hs, T) \nonumber \\
    & x_{0 } \in X_{0 \ell_0}.     \label{eq:peak_hy_traj}
    \end{align}
The optimization variables of \eqref{eq:peak_hy_traj} are the peak time $t$, initial location $\ell_0$, and initial state $x_{0} \in X_{\ell_0}$. The inner maximization runs over all location-objective functions $p_\ell$.

The following assumptions will be posed on problem \eqref{eq:peak_hy_traj}:
\begin{assumption}
The sets $[0,T], \  X_\ell, X_{\ell 0}$ are compact for all $\ell=1..L$.
\end{assumption}
\begin{assumption}
Each function $p_\ell$ is continuous inside $X_\ell$.
\end{assumption}
\begin{assumption}
Each dynamics function $f_\ell(t, x_\ell)$ is Lipschitz over the compact set $[0, T] \times X_\ell$.
\end{assumption}
\begin{assumption}
Trajectories stay in $X_\ell$ for all $t \in [0, T]$ when starting inside $X_{0 \ell}$.
\end{assumption}

\subsection{Measures for Hybrid Systems}

The control and reachability set programs in \cite{shia2014reachable,mohan2016convex,  han2019controller} define measures $\rho_e$ supported over the guard $\Mp{S_e}$ for each transition $e \in \es$. For subsets $A \subset [0, T], \ C_e \subset S_e$ and an initial condition $x_0$, the counting measure $\rho_e$ records the number of times the trajectory, starting from location $\textrm{src}(e)$, enters the patch $C_e$ of the guard $S_e$ with
\begin{equation}
\label{eq:meas_count}
    \rho_e(A \times C_e) = \int_{A} \textrm{card}\left(\lim_{t'\rightarrow t^-} x(t' \mid x_0) \in C_e\right) dt.
\end{equation}
The mass of the counting measure $\rho_e$ is the expected number of times a trajectory will traverse the transition with arc $e$. In a Zeno execution of transition $e$, the mass $\inp{1}{\rho_e}$ will be unbounded, and constraints such as $\inp{1}{\rho_e} \leq N_e$ may be imposed to cap the maximum number of transitions on arc $e$.
Let $X_{0\ell} \subseteq X_\ell$ be a set of initial conditions defined on each space $X_\ell$ in $X$. A distribution of initial conditions over each location is $\mu_{0\ell} \in \Mp{X_{0\ell}}$ for $\ell = 1..L$. Let $T < \infty$ be a final time, and $\mu_{p \ell} \in \Mp{[0, T] \times X_\ell}$ be peak measures supported over each location-space. Trajectories following dynamics $x'(t) = f_\ell(t, x(t))$ in each space $X_\ell$ are tracked by occupation measures $\Mp{[0, T] \times X_\ell}$. Counting measures $\rho_e \in \Mp{S_e}$ are set up over all guards to handle state transitions.
The Liouville equation with guard measures holding for all test functions $v_\ell \in C^1([0, T] \times X_\ell)$ and locations $\ell = 1..L$ is
\begin{align}
        \label{eq:liou_hy} \mu_{p  \ell} &= \delta_0 \otimes\mu_{0 \ell} + \Lie_{f_\ell}^\dagger \mu_\ell\\
        &+\textstyle \sum_{\textrm{src}(e) = \ell} R_{e\#} \rho_e - \textstyle \sum_{\textrm{dst}(e) = \ell} \rho_e. \nonumber
\end{align}
For a location $\ell$ and edge $e$ with $\textrm{src}(e) = \ell$, the pushforward term $R_{e \#}$ in \eqref{eq:liou_hy} should be understood as
\begin{equation}
    \inp{v_{\ell}}{R_{e\#} \rho_e} = \inp{v_{\ell}(t, R_e(x_\ell))}{\rho_e}.
\end{equation}
The mass of the peak measure $\mu_{p\ell}$ is equal to the mass of the initial measure $\mu_{0\ell}$ plus the net flux due to state transitions. 

\subsection{Measure Program}
\label{sec:hybrid_meas}
Problem \eqref{eq:peak_hy_traj} may be relaxed through an infinite-dimensional linear program in occupation measures. The measures $\mu_{0\ell}$ are distributions of initial conditions, and $\rho_e$ are transition counting measures, just as in the Liouville equation \eqref{eq:liou_hy}. The peak measures $\mu_{p \ell}$ are final measures with free terminal time between $t \in [0, T]$. 
The measure program in terms of $(\mu_0, \mu_p, \mu, \rho)$ for hybrid peak estimation  (where $\forall \ell$ and $\forall e$ may be expanded to $\forall \ell = 1..L$ and $\forall e \in \es$) is
\begin{subequations}
\label{eq:peak_meas_hy}
\begin{align}
p^* = & \ \textrm{sup} \quad \textstyle\sum_{\ell=1}^L \inp{p_\ell}{\mu_{p\ell}} & \label{eq:peak_meas_hy_obj} \\
    & \mu_{p  \ell} = \delta_0 \otimes\mu_{0 \ell} + \Lie_{f_\ell}^\dagger \mu_\ell  \label{eq:peak_meas_hy_flow} & & \forall \ell\\
    & \quad \ \, +\textstyle \sum_{\textrm{dst}(e) = \ell} R_{e\#} \rho_e  - \textstyle \sum_{\textrm{src}(e) = \ell}\rho_e  \nonumber\\
    & \textstyle\sum_{\ell=1}^L \inp{1}{\mu_{0 \ell}} = 1 & \label{eq:peak_meas_hy_prob}\\
    &  \inp{1}{\rho_e} \leq N_e & & \forall e \label{eq:peak_meas_hy_Zeno}\\
    & \mu_{\ell}, \ \mu_{p \ell} \in \Mp{[0, T] \times X_\ell} & & \forall \ell\label{eq:peak_meas_hy_peak_occ}\\
    & \mu_{0\ell} \in \Mp{X_{0 \ell}}& &\forall \ell \label{eq:peak_meas_hy_init} \\
    & \rho_{e} \in \Mp{S_e} & & \forall e. \label{eq:peak_meas_hy_guard} 
\end{align}
\end{subequations}
\begin{theorem}
Solutions to \eqref{eq:peak_meas_hy} and \eqref{eq:peak_hy_traj} satisfy $p^* \geq P^*$
\end{theorem}
\begin{proof}
Let $(x(t \mid x_0, \ell_0), \mathcal{T},\mathcal{C})$ be a trajectory from the execution of Algorithm \ref{alg:hybrid_exec} that stops at time $t^* \in [0, T]$, and $\textrm{Loc}(t)$ be the function returning the residing location of $x(t)$ at time $t$. This trajectory may be described by a tuple $(\ell_0, x_0, t^*)$. 
Measures $\forall \ell: \mu_{0\ell}, \mu_{p \ell}, \mu_{\ell}$ and $\forall e: \rho_e$ that are feasible solutions to constraints \eqref{eq:peak_meas_hy_flow}-\eqref{eq:peak_meas_hy_guard} may be formed from the trajectory $x(t)$. The initial measure $\mu_{0\ell}$ is $\delta_{x = x_0}$ for $\ell=\ell_0$ and is the zero measure for $\ell \neq \ell_0$. The peak measure $\mu_{p \ell}$ is $\delta_{t=t^*} \otimes \delta_{x = x(t^* \mid x_0, \ell_0)}$ for $\ell = \textrm{Loc}(t^*)$ and is also the zero measure for all other $\ell$. 
Let $T_\ell$ be the set $T_\ell = \{t \mid t \in [0, t_p^*], \textrm{Loc}(t) = \ell\}$ of times where $x(t)$ is in location $\ell$.
Each relaxed occupation measure $\mu_\ell$ may respectively be set to the occupation measure of $t \mapsto (t, x(t \mid x_0, \ell_0))$ in the times $t \in T_\ell$.
If the transition with edge $e \in \es$ is traversed $\mathcal{N}_e$ times along the trajectory $x(t)$ at points $\{(t_i^e, x_{i}^e)\}_{i = 1}^{\mathcal{N}_e}$ for $x_{i}^e \in X_{\textrm{src}(e)}$, the guard measure $\rho_e$ may be defined as $\rho_e = \sum_{i=1}^{\mathcal{N}_e} \delta_{t=t_i^e} \otimes \delta_{x=x_{i}^e}.$ 
The objective $p^*$ is an upper bound on $P^*$ because a set of measures $(\mu_{0\ell}, \mu_{p \ell}, \mu_{\ell}, \rho_e)$ constructed from every trajectory $x(t)$ satisfy the constraints of \eqref{eq:peak_meas_hy} with objective $P^*$. 
\end{proof}





\begin{remark}
\label{rmk:peak_below}
Setting a peak objective to $p_\ell(x) = -\infty$ is equivalent to constraining $\mu_{p \ell}$ to the zero measure, because trajectories to maximize $p(x)$ will not terminate in location $\ell$.  Likewise, a measure $\mu_{0\ell} \in \Mp{X_{0\ell}}$ where $X_{0\ell}=\varnothing$ is the zero measure.
\end{remark}

\begin{theorem}
\label{thm:bounded_mass} All measures involved in a solution to \eqref{eq:peak_meas_hy} are bounded.
\end{theorem}
\begin{proof}
Sufficient conditions for a measure to be bounded are that its mass is finite and its support is compact. This setting satisfies the compact support requirement. 

Given that all measures $(\mu_0, \mu_p, \mu, \rho)$ are nonnegative, their masses will also be nonnegative numbers. The mass of the transition measures $\rho$ are upper bounded by the Zeno constraints \eqref{eq:peak_meas_hy_Zeno} under the assumption that all $N_e$ are finite. Constraint \eqref{eq:peak_meas_hy_init} upper bounds each mass $\inp{1}{\mu_{0 \ell}}$. For each location $\ell$, choosing a test function $v_\ell(t, x_\ell) = 1$ for Liouville equation \eqref{eq:peak_meas_hy_flow} yields
\begin{align}
\label{eq:mass_liou_1}
    \inp{1}{\mu_{p\ell}} &= \inp{1}{\mu_{0\ell}} + \textstyle\sum_{\textrm{dst}(e) = \ell} \inp{1}{\rho_e } - \textstyle \sum_{\textrm{src}(e) = \ell}\inp{1}{\rho_e}.\\
\intertext{Every term on the right-hand side of \eqref{eq:mass_liou_1} is finite and $\inp{1}{\mu_{p\ell}} \geq 0$ by measure nonnegativity, so each peak measure $\mu_{p\ell}$ has bounded mass. Utilizing a test function of $v_\ell(t,x_\ell) = t$ with $\Lie_{f_\ell}t = 1$ results in}
\label{eq:mass_liou_t}
    \inp{t}{\mu_{p\ell}} &= \inp{1}{\mu_{\ell}} + \textstyle\sum_{\textrm{dst}(e) = \ell} \inp{t}{\rho_e } - \textstyle \sum_{\textrm{src}(e) = \ell}\inp{t}{\rho_e}.
\end{align} 

The terms $\inp{t}{\mu_{p\ell}}, \inp{t}{\rho_e} $ are all finite due to bounded masses and compact support, so the occupation measures $\mu_{\ell}$ also have finite mass and are bounded.
\end{proof}

\begin{theorem}
\label{thm:no_relaxation}
The objectives in \eqref{eq:peak_hy_traj} and \eqref{eq:peak_meas_hy} will satisfy $p^* = P^*$ when $[0, T] \times \prod_\ell X_\ell$ is compact, each $f_\ell$ is Lipschitz, and $p^*$ is bounded above. 
\end{theorem}
\begin{proof}
This statement may be proved by extending arguments from  \cite{zhao2019optimal}. Theorem 17 of \cite{zhao2019optimal} states there is no relaxation gap in measure \acp{LP} of an optimal control program with appropriate assumptions, extending the \ac{ODE} result of \cite{lewis1980relaxation}. Free final time is already accounted for in \cite{zhao2019optimal} by reference to Remark 2.1 of \cite{lasserre2008nonlinear}. The \ac{ODE} problem in \cite{lewis1980relaxation} can handle initial conditions lying in a set $X_0$, so the method in \cite{zhao2019optimal} can similarly work with sets of initial conditions $\{X_{0\ell}\}_{\ell=1}^L$ as demonstrated by \cite{zhao2019switching}. The work in \cite{zhao2019switching} has `switching' costs (possibly differing running and terminal costs in each location), which is realized by the costs $p_\ell$. The final modification between this work and \cite{zhao2019optimal} is that problem \eqref{eq:peak_hy_traj} has finite Zeno caps $N_e$, while Assumption 3 of \cite{zhao2019optimal} forbids Zeno trajectories. The allowance for free terminal time permits consequence 4 of Theorem 12 of \cite{zhao2019optimal} to read that there exists a constant $C$ such that $\sum_e \inp{1}{\rho_e} \leq \sum_e N_e = C$. The three modifications of \cite{zhao2019optimal} (free terminal time, multiple initial conditions, Zeno caps) are all cleared, so $p^* = P^*$ under the compactness and Lipschitz assumptions.

\end{proof}


\subsection{Function Program}
\label{sec:hybrid_cont}

The measure program \eqref{eq:peak_meas_hy} is dual to an infinite-dimensional linear program in continuous functions. The Lagrangian $\scL$ of problem \eqref{eq:peak_meas_hy} with dual variables $v_\ell \in C^1([0, T] \times X_\ell), \ \gamma \in \Rn, \alpha \in \Rn_+^{\abs{\es}}$ is
\begin{align}
    \label{eq:peak_lagrangian}
    \scL &= \textstyle\sum_{\ell=1}^L \inp{p_\ell}{\mu_{p \ell}} + \inp{v_\ell(t,x)}{\delta_0 \otimes\mu_{0 \ell} + \Lie_{f_\ell}^\dagger \mu_\ell}  \\
    &+\inp{v_\ell(t, x)}{\textstyle \sum_{\textrm{dst}(e) = \ell} R_{e\#} \rho_e - \textstyle \sum_{\textrm{src}(e) = \ell} \rho_e - \mu_{p \ell}} \nonumber\\
    &+ \gamma(1 - \textstyle\sum_{\ell=1}^L\inp{1}{\mu_{0 \ell}}) + \sum_{e \in \es}\alpha_e (N_e  -\inp{1}{\rho_e}). \nonumber 
\end{align}


The dual function program of \eqref{eq:peak_meas_hy} is
\begin{subequations}
\label{eq:peak_cont_hy}
\begin{align}
    d^* = & \inf_{\gamma, \alpha, v} \quad \sup_{\mu_{0\ell}, \mu_{p \ell}, \mu_{\ell}, \rho_e} \scL \nonumber \\
  d^* = &\ \inf_{\gamma \in \Rn, \ \alpha \in \Rn_+^{\abs{\es}}} \quad \gamma + \textstyle\sum_{e \in \es} N_e \alpha_e & & \\
    & \forall \ell: \ \forall x_\ell \in  X_{0\ell}: \nonumber\\
    & \qquad \gamma \geq v_\ell(0, x_\ell)  & \label{eq:peak_cont_hy_init}\\
    & \forall \ell: \ \forall (t, x_\ell)\in [0, T] \times X_\ell: \nonumber\\
    & \qquad 0 \geq \Lie_{f_\ell} v_\ell(t, x_\ell)& &  \label{eq:peak_cont_hy_flow}\\
    & \forall e: \ \forall (t,x_{\textrm{src}(e)}) \in [0, T] \times X_{\textrm{src}(e)}: \nonumber\\
    & \qquad v_{\textrm{src}(e)}(t, x_{\textrm{src}(e)}) - v_{\textrm{dst}(e)}(t, R_e(x_{\textrm{src}(e)}))\geq -\alpha_e    \label{eq:peak_cont_hy_jump}\\
    & \forall \ell: \ \forall (t, x_\ell) \in [0, T] \times X_\ell: \nonumber\\
    & \qquad v_\ell(t, x_\ell) \geq p_\ell(x_\ell) & \label{eq:peak_cont_hy_p} \\
    & \forall \ell: \ v_\ell(t,x_\ell) \in C^1([0, T]\times X_\ell) \label{eq:peak_cont_hy_v}. \qquad  
    \end{align}
\end{subequations}

The dual variables $v_\ell$ are auxiliary functions that decrease along trajectories \eqref{eq:peak_cont_hy_flow} and along transitions \eqref{eq:peak_cont_hy_jump}. The auxiliary functions upper bound the location-costs by \eqref{eq:peak_cont_hy_p}. 
The dual variable $\alpha_e$ will be zero if transition $e$ is traveled at most $N_e-1$ times (complementary slackness of  \eqref{eq:peak_meas_hy_guard}).

\begin{theorem}
\label{thm:strong_duality}
Programs \eqref{eq:peak_meas_hy} and \eqref{eq:peak_cont_hy} will possess equal objectives $p^*=d^*$ when each $X_\ell$ is compact and $(T, N_e)$ are each finite.
\end{theorem}
\begin{proof}
$p^*=d^*:$
Strong duality follows by arguments from Theorem 2.6 of \cite{tacchi2022convergence}, specifically from boundedness of measures (Theorem \ref{thm:bounded_mass}) and compactness (Assumption A2).
\end{proof}

\subsection{Linear Matrix Inequality Program}
\label{sec:hybrid_lmi}

The Moment-\ac{SOS} hierarchy is a method to produce upper bounds to measure \acp{LP} by a sequence of Linear Matrix Inequalities (\acp{LMI}) of increasing size
\cite{lasserre2009moments}. Let $X \in \Rn^n$ be a basic semialgebraic set $X = \{x \mid g_i(x) \geq 0, i=1..N_c\}$, which is the locus of a finite number of finite-degree polynomial inequality constraints. 
An $\alpha$-moment of a measure $\mu \in \Mp{X}$ for $\alpha \in \Nn^{n}, \ \beta \in \Nn$ is $\mathbf{y_{a}} = \inp{x^\alpha}{ \mu}$. To each moment sequence $\mathbf{y}$, there is an associated Riesz linear functional $\mathbb{L}_{\mathbf{y}}$ acting as $\mathbb{L}_{\mathbf{y}}[ \sum_{\alpha, \beta} c_{\alpha  \beta} x^\alpha]\rightarrow  \sum_{\alpha} c_{\alpha}\mathbf{y}_{\alpha}$.

Assume that each polynomial $g_i(x) = \sum_\gamma g_{i\gamma} x^\gamma$ in the definition of $X$ has a finite degree $d_i$.
If the set $X$ satisfies an Archimedean condition (all compact sets may be made Archimedean by adding a redundant ball constraint) \cite{putinar1993compact}, then necessary and sufficient conditions for the sequence $\mathbf{y}$ of putative moments (pseudo-moments) up to degree $2d$ to be moments of a measure $\mu \in \Mp{X}$ are that the following matrices indexed by monomials $\alpha, \beta \in \Nn^n$ are \ac{PSD}:
\begin{equation}
\label{eq:moment_loc}
    \M_d(\mathbf{y})_{\alpha \beta} = \mathbf{y}_{\alpha + \beta}, \quad \M_{d-d_i}(g_{i} \mathbf{y})_{\alpha \beta}  =  \textstyle\sum_{\gamma} g_{i \gamma} \mathbf{y}_{\alpha + \beta + \gamma}.
\end{equation}
The measure $\mu$ is referred to as the \textit{representing measure} of the pseudo-moments $\mathbf{y}$.
The symbol $\M_d(X \mathbf{y})$ will denote a block-diagonal matrix formed by the matrices in \eqref{eq:moment_loc}.

The basic semialgebraic sets containing measures in \eqref{eq:peak_meas_hy} are
\begin{align}
    \label{eq:peak_sets}
    \forall \ell: & & X_\ell &= \{x_\ell \mid g_{\ell i}(x_\ell )\geq 0 \mid \ i = 1..N_c^{\ell}\} \nonumber\\
    \forall \ell:  & & X_{0 \ell} &= \{x_\ell \mid  g_{0 \ell i}(x_\ell )\geq 0 \mid \ i = 1..N_c^{0\ell}\}\\
    \forall e:  & & S_e &= \{x_{\textrm{src}(e)} \mid g_{e i}(x_{\textrm{src}(e)})\geq 0 \mid \ i = 1..N_c^{e}\}. \nonumber
\end{align}

Polynomials $g_{\ell i}(x), \ g_{0 \ell i}(x_\ell), \  g_{e i}(x_{\textrm{src}(e)})$ have finite degrees $d_{\ell i}, \ d_{0\ell i}, \  d_{e i}$ respectively for each  $i, \ell, e$ as appropriate.
Let $(\mathbf{y}^{0 \ell}, \mathbf{y}^{p \ell}, \mathbf{y}^{\ell}, \mathbf{r}^e)$ be pseudo-moments of the measures $(\mu_{0 \ell}, \mu_{p \ell}, \mu_p, \rho_e)$.
The Liouville equation \eqref{eq:peak_meas_hy_flow} may be expressed as a collection of affine constraints in the pseudo-moments. Substituting the test function $v(t, x_\ell) = x_\ell^\alpha t^\beta$ into \eqref{eq:peak_meas_hy_flow} yields a relation for each $\alpha \in \Nn^{n_\ell}, \ \beta \in \Nn, \ \ell \in 1..L$:
\begin{align}
    \label{eq:liou_lmi_hy}
        0 &= -\inp{x_\ell^\alpha t^\beta}{\mu_{p \ell}} + \inp{x_\ell^\alpha t^\beta}{ \delta_{t=0}\otimes \mu_{0 \ell}} + \inp{\Lie_{f_\ell}x_\ell^\alpha t^\beta}{\mu_\ell} \nonumber\\
        &+\textstyle \sum_{\textrm{dst}(e) = \ell} \inp{R_e(x_\ell)^\alpha t^\beta}{\rho_e} - \textstyle \sum_{\textrm{src}(e) = \ell}\inp{x_\ell^\alpha t^\beta}{\rho_e}.
\end{align}
The expression $\textrm{Liou}^\ell_{\alpha \beta}(\mathbf{y}^{0 \ell}, \mathbf{y}^{p \ell}, \mathbf{y}^{\ell}, \mathbf{r}^{\es_\ell} ) = 0$ may be defined to abbreviate the affine constraint in pseudo-moments induced by \eqref{eq:liou_lmi_hy}, where $\es_\ell = \{e \in \es \mid \textrm{src}(e) = \ell \textrm{ or } \textrm{dst}(e) = \ell\}$ is the set of arcs including location $\ell$. For a constant degree $d \in \Nn$, define the quantities $d_\ell' = d + \ceil{\textrm{deg}f_\ell/2}-1$ and $k_e = \deg{R_e}$.
The degree-$d$ \ac{LMI} relaxation of \eqref{eq:peak_meas_hy} with variables $(\mathbf{y}^{0 \ell}, \mathbf{y}^{p \ell}, \mathbf{y}^{\ell}, \mathbf{r}^{e} )$ is

\begin{subequations}
\label{eq:peak_lmi_hy}
\begin{align}
    p^*_d = & \textrm{max} \quad \textstyle \sum_{\ell} \textstyle\sum_{\alpha} p_{\ell \alpha} \mathbf{y}_{\alpha}^{p \ell} \label{eq:peak_lmi_hy_obj} \\
    & \textstyle \sum_\ell \mathbf{y}^{0\ell}_0 = 1 \\
     \forall \ell: \ & \alpha \in \Nn^{n_\ell}, \beta \in \Nn, \ \abs{\alpha} + \abs{\beta} \leq 2d \nonumber\\
    & \qquad  \textrm{Liou}_{\alpha \beta}^\ell(\mathbf{y}^{0 \ell}, \mathbf{y}^{p \ell}, \mathbf{y}^{\ell}, \mathbf{r}^{\es_\ell} ) = 0 \textrm{ by \eqref{eq:liou_lmi_hy}} \label{eq:peak_lmi_hy_flow}\\
    \forall e: \ & \mathbf{y}^{e}_0 \leq N_e \label{eq:peak_lmi_hy_Zeno}  \\
    \forall \ell: \ &  \M_d(X^{0\ell}\mathbf{y}^{0 \ell}), \ \M_d([0,T]\times X^{\ell}\mathbf{y}^{p \ell}), \  \M_{d'_\ell}([0,T] \times X^\ell\mathbf{y}^{\ell})
     \succeq 0  \label{eq:peak_lmi_hy_psd} \\
    \forall e: \ & \M_{k_e d_e}(S_e\mathbf{r}^{e})
     \succeq 0. \qquad  \label{eq:peak_lmi_hy_psd_guard} 
\end{align}
\end{subequations}

The affine constraints \eqref{eq:peak_lmi_hy_flow}-\eqref{eq:peak_lmi_hy_Zeno} implement a truncation of constraints \eqref{eq:peak_meas_hy_Zeno}-\eqref{eq:peak_meas_hy_Zeno} in terms of finite-length pseudo-moments. Constraints \eqref{eq:peak_lmi_hy_psd}-\eqref{eq:peak_lmi_hy_psd_guard} ensure that there exist representing measures for the pseudo-moments. Solutions to the \ac{SDP} generated from the \ac{LMI} \eqref{eq:peak_lmi_hy} by  raising the degree $d$ will form a chain of upper bounds  $p^*_d \geq p^*_{d+1} \geq \ldots \geq p^*$.

\begin{theorem}
\label{thm:lmi_convergence}
The sequence of upper bounds will satisfy $\lim_{d \rightarrow \infty} p_d^* = P^*$ when $\forall \ell=1..L:$ $[0, T] \times  X_\ell$ and $X_{0 \ell}$ are Archimedean, $f_\ell(t,x)$ are
 polynomial, and $\forall e: N_e$ are  finite.
\end{theorem}
\begin{proof}
The upper bound sequence will converge to $p^*$ when all sets are Archimedean, there exists an interior point to constraints \eqref{eq:peak_meas_hy_flow}-\eqref{eq:peak_meas_hy_guard}, and all measures $(\mu_{0\ell}, \mu_{p \ell}, \mu_{\ell}, \rho_e)$ have bounded moments (\cite{tacchi2022convergence}, Theorem 5 of \cite{trnovska2005strong} and Theorem 4.4 of \cite{lasserre2009moments}).

Let $x_0$ be an initial point starting in some nonempty location $X_\ell$. The set of measures where $\mu_{0\ell} = \delta_{x=x_0}, \ \mu_{p\ell} = \delta_{t=0}\otimes \delta_{x=x_0}$ and all other measures are the zero measure is an interior point to \eqref{eq:peak_meas_hy_flow}-\eqref{eq:peak_meas_hy_guard} (trajectory starting at $x_0$ with zero elapsed time).
Given that each $[0, T] \times X_\ell$ is compact, it is sufficient that all measures have bounded masses in order for the measures to have bounded moments.  The masses of $\rho_e$ are each upper bounded by the finite quantity $N_e$ by constraint \eqref{eq:peak_meas_hy_guard}, and the sum of the masses of $\mu_{0\ell}$ are upper bounded by 1 through \eqref{eq:peak_meas_hy_prob}. The sum of constraint \eqref{eq:peak_meas_hy_flow} with test function $v_\ell = 1$ along all $\ell$ is $\sum_{\ell = 1}^L \inp{1}{\mu_{p  \ell}} =  \textstyle \sum_{\ell = 1}^L \inp{1}{\mu_{0  \ell}} = 1$, so each mass of  $\mu_{p \ell}$ is finite. Lastly, the use of a test function of $v_\ell=t$ on each Liouville equation in \eqref{eq:peak_meas_hy_flow} yields the finite expression $\inp{1}{\mu_\ell} =  \inp{t}{\mu_p} -\sum_{\textrm{dst}(e)=\ell} \inp{t}{\rho_e} + \sum_{\textrm{src}(e)=\ell} \inp{t}{\rho_e}$. The sequence of upper bounds will therefore converge to $p^*$ as $d \rightarrow \infty$ with $p^* = P^*$ from Theorem \ref{thm:no_relaxation}.
\end{proof}

 The sizes of the moment matrices in problem \eqref{eq:peak_lmi_hy} are listed in Table \ref{tab:mom_size}. The computational complexity of numerical \ac{SDP} solvers scale in a polynomial manner with the size of the largest \ac{PSD} matrix \cite{lasserre2008nonlinear}. These \ac{PSD} matrix sizes may be reduced if extant structure such as symmetry, quotient, or sparsity structure is present in  \eqref{eq:peak_meas_hy}.
\begin{table}[h]
\caption{\label{tab:mom_size}Sizes of moment matrices in \ac{LMI} \eqref{eq:peak_lmi_hy}}
        \centering
        \begin{tabular}{c c c c c}
             Moment&  $\M_d(\mathbf{y}^{0\ell})$ & $\M_d(\mathbf{y}^{p\ell})$ & $\M_{d'_\ell}(\mathbf{y}^{\ell})$ & $\M_{d k_e}(\mathbf{r}^{e})$ \\
             Size & $\binom{n_\ell+d}{d}$ & $\binom{1+n_\ell+d}{d}$ & $\binom{1 + n_\ell+d'_\ell}{d'_\ell}$ & $\binom{1+n_\ell + k_e d}{k_e d}$         \end{tabular}
    \end{table}


\begin{remark}
Guards with codimension-1 sets $S_e$ may replace their \ac{PSD} localizing constraints  with linear equality constraints or quotient ring reductions in $\mathbf{r}^e$.
\end{remark}



\begin{remark}
Algorithm 1 of \cite{miller2020recovery} may be used to attempt extraction of near-optimal trajectories if the moment matrices $\forall \ell: \ \M_{d}(\mathbf{y}^{0 \ell}), \M_{d}(\mathbf{y}^{p \ell})$ are low-rank.
\end{remark}

\section{Extensions}

\label{sec:extensions}
This section details extensions to the previously presented peak estimation framework for hybrid systems.
\subsection{Safety}

This section verifies safety of hybrid system trajectories with respect to a group of unsafe sets, based on prior (\ac{ODE}) work in \cite{miller2021distance}. Let $X_{u \ell} = \{x_\ell \mid p_{\ell i}(x_\ell) \geq 0, \ i = 1..N_u\}$ be an unsafe basic semialgebraic set for each location $\ell = 1..L$. 
Letting $c_\ell(x_\ell, y_\ell)$ be a distance function (cost) with a point-unsafe-set distance
\begin{equation}
    c_\ell(x_\ell; X_{u\ell}) = \inf_{y_\ell \in X_{u\ell}} c_\ell(x_\ell, y_\ell),
\end{equation} the distance estimation problem for hybrid systems is
\begin{equation}
    \begin{aligned}
    Q^* =& \inf_{ \ell' \in 1..N_u, t \in [0, T], \ell_0, x_0} c(x(t \mid x_0); X_{u_{\ell'}}) \\
    & \textrm{Dynamics follow Algorithm \ref{alg:hybrid_exec} with input } (\ell_0, x_0, \hs, T) \nonumber \\
    & x_{0 } \in X_{0 \ell}.     \label{eq:dist_hy_traj}
    \end{aligned}
\end{equation}

Following the procedure from \cite{miller2021distance}, a joint-measure $\eta_\ell(x_\ell, y_\ell) \in \Mp{X_\ell \times X_{u \ell}}$ is added for each unsafe set. The distance objective in \eqref{eq:dist_hy_traj} is replaced with an equivalent expectation over the joint probability measure $\inp{c_\ell(x_\ell, y_\ell)}{\eta_\ell}$.

The measure program for distance estimation with variables $(\mu_{p \ell}, \mu_{0 \ell}, \mu_{\ell},\eta_\ell,  \rho_e)$ is
\begin{subequations}
\label{eq:peak_meas_dist}
\begin{align}
q^* = & \ \textrm{inf} \quad \textstyle \sum_{\ell=1}^L \inp{c_\ell(x_\ell, y_\ell)}{\eta_\ell} \label{eq:peak_meas_dist_obj} \\
    & \pi^{x_\ell}_\# \eta_\ell = \pi^{x_\ell}_\# \mu_{p \ell} \ \forall \ell \label{eq:dist_meas_marg_bound}\\
    & \textrm{Constraints \eqref{eq:peak_meas_hy_flow}-\eqref{eq:peak_meas_hy_Zeno}} \\
    & \textrm{Variables from \eqref{eq:peak_meas_hy_peak_occ}-\eqref{eq:peak_meas_hy_guard}}\\
    & \forall \ell: \quad \eta_\ell(x_\ell, y_\ell) \in \Mp{X_\ell \times X_{u \ell}}  \label{eq:dist_meas_eta}.
\end{align}
\end{subequations}

\subsection{Uncertainty}
Peak estimation for hybrid systems may be applied to systems with  uncertainty, extending the \ac{ODE} case in \cite{miller2021uncertain}. 
Let $ W_\ell \subset \Rn^{N_{w\ell}}$ 
be a compact set of  time-dependent disturbances for each location. 
Each location obeys dynamics $\dot{x}_\ell = f(t, x_\ell(t), w_\ell(t)),  \  \forall t, \ell: w_\ell(t) \in W_\ell$, 
in which there is no prior assumption of continuity on the process $w(\cdot)$. The uncertainty act as adversarial optimal controls attempting to raise the  peak functions $(p_\ell)$.

Uncertainty in this manner may be realized by adjusting the Liouville equation in \eqref{eq:peak_meas_hy_flow} and occupation measure definitions in \eqref{eq:peak_meas_hy_peak_occ} (where $w_\ell(t)$ acts as a Young measure/relaxed control \cite{young1942generalized}) as in
\begin{subequations}
\label{eq:uncertainty_mods}
\begin{align}
        & \mu_{p  \ell} = \delta_0 \otimes\mu_{0 \ell} + \pi^{tx}_\# \Lie_{f_\ell}^\dagger \mu_\ell  \label{eq:peak_meas_unc_hy_flow} & & \forall \ell\\
& \quad \ \, +\textstyle \sum_{\textrm{dst}(e) = \ell} R_{e\#} \rho_e  - \textstyle \sum_{\textrm{src}(e) = \ell}\rho_e  \nonumber\\
& \mu_{\ell} \in \Mp{[0, T] \times X_\ell \times W_\ell} & & \forall \ell.\label{eq:peak_meas_unc_hy_peak_occ}
\end{align}
\end{subequations}

A particular form of time-dependent uncertainty is switching/polytopic structure. If the system model is $\dot{x}_\ell = \sum_k^{N_s} w_{k \ell} f_{k \ell}(t, x)$ for $N_s$ switching modes and $w_{k \ell} \geq 0, \sum_k w_{k \ell} = 1$, then the Liouville equation in \eqref{eq:uncertainty_mods} may be expressed for occupation measures $\mu_{k \ell} \in \Mp{[0, T] \times X}$ as
\begin{subequations}
\label{eq:uncertainty_switching}
\begin{align}
        & \mu_{p  \ell} = \delta_0 \otimes\mu_{0 \ell} + \textstyle \sum_k^{N_s} \Lie_{f_{k \ell}}^\dagger \mu_{k \ell}  \label{eq:peak_meas_switch_hy_flow} & & \forall \ell\\
& \quad \ \, +\textstyle \sum_{\textrm{dst}(e) = \ell} R_{e\#} \rho_e  - \textstyle \sum_{\textrm{src}(e) = \ell}\rho_e  \nonumber\\
& \mu_{\ell} \in \Mp{[0, T] \times X_\ell} & & \forall \ell. \label{eq:peak_meas_switch_hy_peak_occ}
\end{align}
\end{subequations}

Time-independent uncertainty restricted to a compact set $\Theta \subset \Rn^{N_\theta}$ may also be added by adjoining to dynamics a new state $\dot{x}_\ell = f(t, x_\ell(t), \theta, w_\ell(t)), \ \dot{\theta} = 0$. This new state $\theta$ is preserved between transition jumps, inducing lifted reset maps $\tilde{R}_{s \rightarrow t}(x_s, \theta) = (R(x_t), \theta)$.

\section{Numerical Examples}
\label{sec:examples}


Experiments are available at \url{https://github.com/Jarmill/hybrid_peak_est}, and were written in MATLAB 2021a. Dependencies include Gloptipoly3 \cite{henrion2003gloptipoly}, Yalmip interface \cite{lofberg2004yalmip}, and Mosek \cite{mosek92}. All experiments were run on an 
2.30 GHz Intel i9 CPU with 64.0 GB of RAM.

\subsection{Two-Mode}

This system is a modification of Example 2 of \cite{prajna2004safety} to ensure improved numerical conditioning. The two locations correspond to modes of `No Control' ($\ell=1$) and `Control' ($\ell=2$) with dynamics
\begin{subequations}
\label{eq:two_mode_deterministic}
\begin{align}
    f_1(t, x) &= [x_2; -x_1+x_3; x_1 + (1+x_3)^2 (2 x_2 + 3 x_3)] \nonumber \\
    f_2(t, x) &= [x_2; -x_1+x_3; -x_1 - (2 x_2 + 3 x_3)].\label{eq:two_box_dynamics}
\end{align}
\end{subequations}

\subsubsection{Two-Mode: Standard}

Trajectories start in the initial set $X_{01} = \{x \mid \norm{x}_2^2 = 0.2^2\}$ ($X_{02} = \varnothing$), and evolve for a time horizon of $T=20$. The transition edges are $\es = \{(1, 2), (2, 1)\}$ with guard surfaces
\begin{align}
    S_{(1,2)} &= \{x \mid x_1^2/4 + x_2^2 + x_3^2 = 1.5^2\}\label{eq:two_box_guards} \\
    S_{(2,1)} &= \{x \mid x_1^2 + x_2^2 + x_3^2 = 0.2^2\}, \nonumber
\end{align}
and each transition has a trivial reset map $R_e(x) = x$. The Zeno caps used in simulation were $N_{(1,2)}=N_{(2,1)}=5$ with total spaces of $X_1 = X_2 = [-1.5, 1.5]^3$. Figure \ref{fig:two_box_state} plots system trajectories in location 1 (left) and 2 (right), starting from the initial set $X_0$ (gray region).
The peak estimation task for this system is to upper bound extreme values of $p_2(x)=x_1^2$ along system trajectories ($p_1(x)=-\infty$). Solving the \ac{SDP} generated from \ac{LMI} \eqref{eq:peak_lmi_hy} at orders 1-5 produces the sequence of upper bounds,
$p^*_{1:5} = [2.250, 0.6514, 0.4643, 0.4076, 0.3958].$
\begin{figure}[ht]
    \centering
    \includegraphics[width=0.75\linewidth]{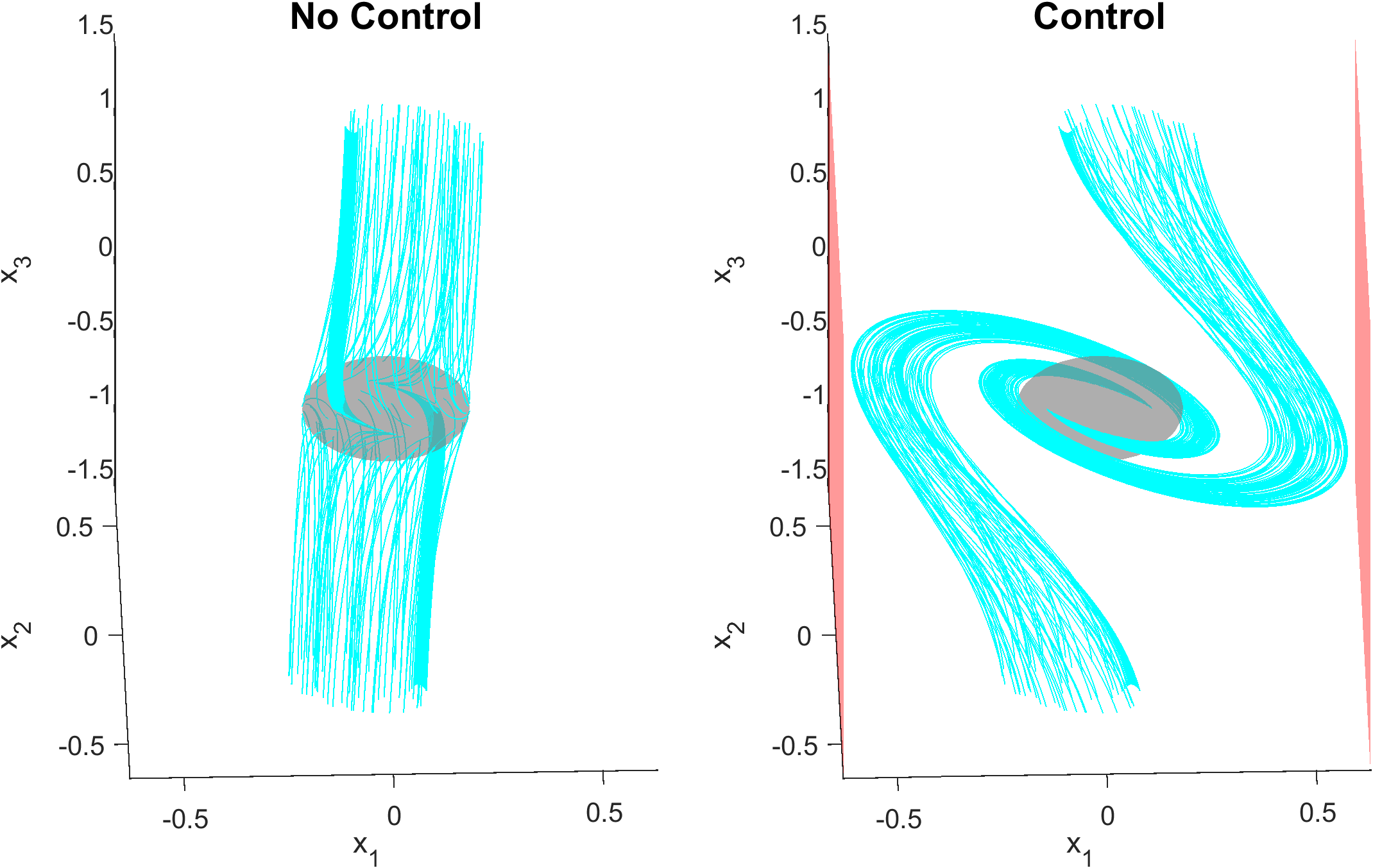}
    \caption{Deterministic Two-Mode Bound of $x_1^2 \leq 0.3958 = p_5^*$}
    \label{fig:two_box_state}
\end{figure}

\subsubsection{Two-Mode: Distance Estimation}
Distance estimation is conducted for the deterministic two-mode system \eqref{eq:two_mode_deterministic} with respect to the half-sphere unsafe set
\begin{equation}
    \label{eq:two_mode_unsafe}
    X_u = \{x \mid 0.4^2 \geq  (x_1 + 0.5)^2 + (x_2 + 0.5)^2 + (x_3 - 0.5)^2, \ x_3 \geq 0.5\}.
\end{equation}
The distance penalty $c(x, y) = \norm{x-y}^2_2$ is used in locations $\ell=1, 2$ with the unsafe set $X_u$. \ac{SDP} lower bounds for the distance $\min_\ell \min_{y \in X_u} \norm{x_{\ell} - y}_2^2$ (via \eqref{eq:peak_meas_dist}) are $p^*_{1:5} = [0, 0, 0, 2.799\times 10^{-3}, 7.942\times 10^{-3}]$.
The output of distance estimation is plotted in Figure \ref{fig:two_mode_distance}. The solid red half-sphere is the set $X_u$, and the corona surrounding $X_u$ is the set of all points with an $L_2$ distance at most $0.0891   = \sqrt{p_5^*}$ away from $X_u$.
\begin{figure}[ht]
    \centering
    \includegraphics[width=0.75\linewidth]{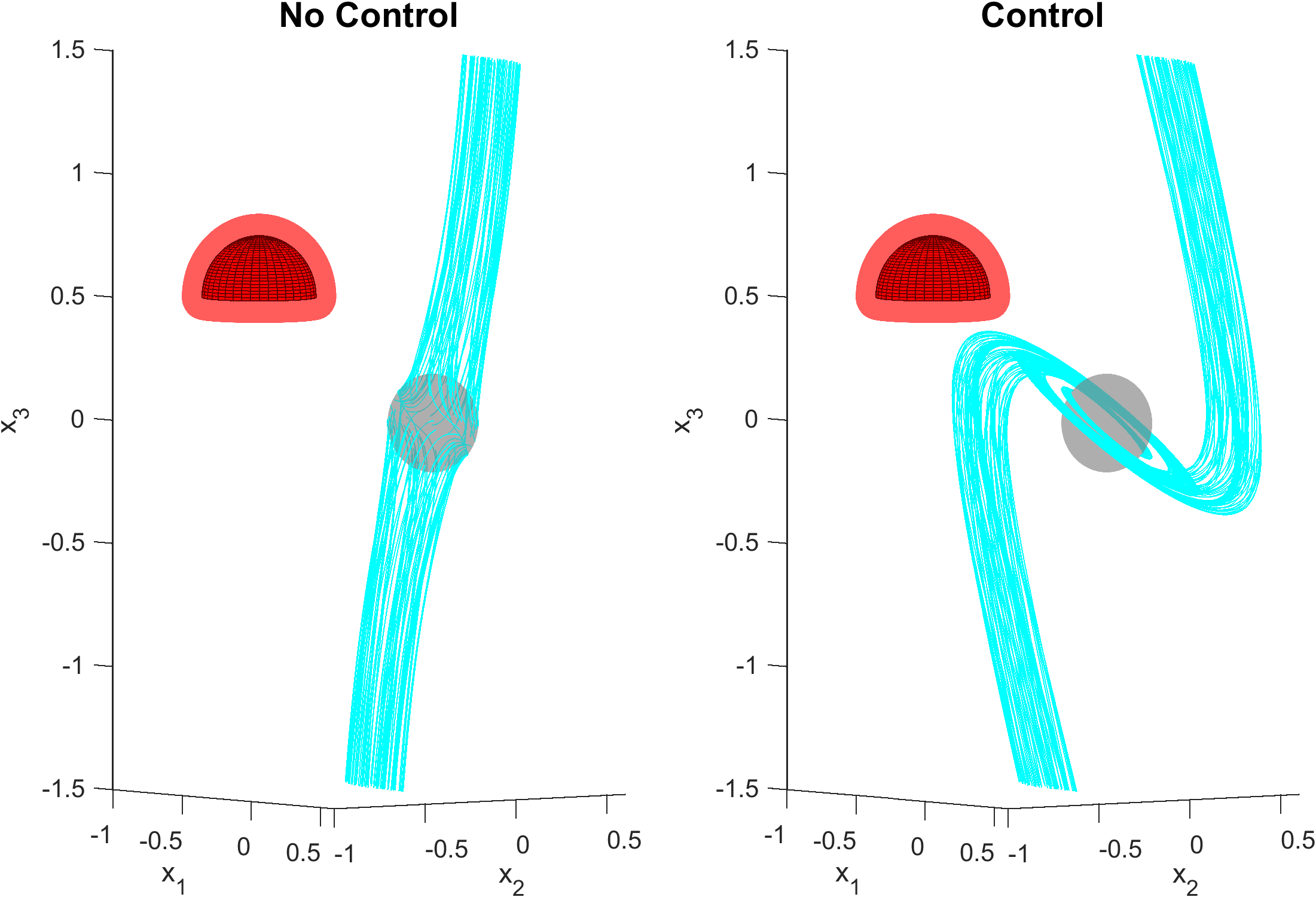}
    \caption{Deterministic Two-Mode Distance Bound of $\min_{y \in X_u}\norm{x-y}_2 \leq 0.0891   = \sqrt{p_5^*}$}
    \label{fig:two_mode_distance}
\end{figure}




\subsubsection{Two-Mode: Uncertainty}
Time-dependent uncertainty may be added to dynamics in \eqref{eq:two_mode_deterministic} by defining a process $w(t) \in [-1, 1]$ under the dynamics $\tilde{f}_\ell(t, x) = f_\ell(t, x) + [0;0;w]$. The \ac{SDP} bounds for $x_1^2$ when $w$ is realized as switching-type uncertainty is 
$p^*_{1:5} = [2.250, 1.4029, 1.0350, 0.9790, 0.9660].$ System trajectories and the order-5 bound of this noisy system are plotted in Figure \ref{fig:two_mode_noisy}.

\begin{figure}[!h]
    \centering
    \includegraphics[width=0.75\linewidth]{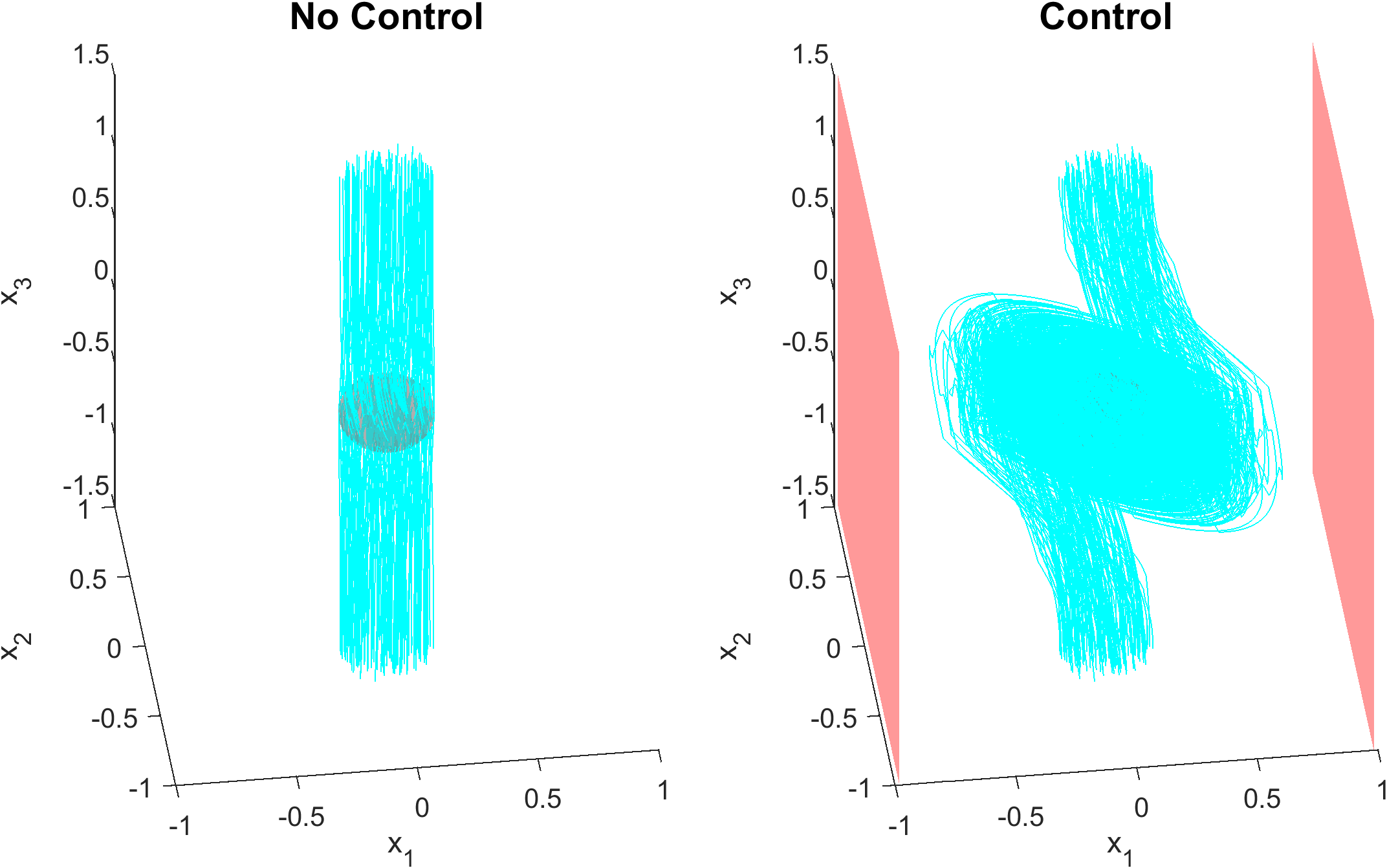}
    \caption{Noisy Two-Mode Bound of $x_1^2 \leq 0.9660 = p_5^*$}
    \label{fig:two_mode_noisy}
\end{figure}

\subsection{Right-Left Wrap}
This example has a single location $X = [-1, 1]^2$ with nontrivial reset maps. The dynamics in the single location are
\begin{equation}
\label{eq:rl_dynamics}
    \dot{x} =\begin{bmatrix}-x_2 + x_1*x_2 + 0.5 \\
         -x_2 - x_1 + x_1^3\end{bmatrix}.
\end{equation}
System \eqref{eq:rl_dynamics} has a stable equilibrium point at $(-0.8128, 0.2758)$ and a saddle point at $(-0.4288, 0.3499)$. The following right$\rightarrow$top and left$\rightarrow$bottom transitions are defined with Zeno caps of $N=5$

\begin{align}
    S_{\textrm{right}\rightarrow \textrm{top}} &= \{x \mid x_1 = 1, \ x_2 \in [-1, 1]\} &  R_{\textrm{right}\rightarrow \textrm{top}} &= [x_2, x_1]^T \label{eq:rl_transitions}\\
    S_{\textrm{left}\rightarrow \textrm{bottom}} &= \{x \mid x_1 = -1, \ x_2 \in [-1, 1]\} & R_{\textrm{left}\rightarrow \textrm{bottom}} &= [1-x_2, x_1]^T.\nonumber
\end{align}
The set $X$ is invariant under these state transitions.
A peak estimation task to maximize $p(x) = -(x_1+0.5)^2 + (x_2+0.5)^2$ is defined on system dynamics \eqref{eq:rl_dynamics} starting from the initial set $X_0 = \{x \mid 0.2^2 = (x_1 - 0.5)^2 + (x_2 + 0.3)^2 \}$ for a $T = 5$ time horizon. \ac{SDP} upper bounds for this objective are $p^*_{1:6} = [0,0,-0.3644, -0.5259, -0.5659, -0.5721]$.

Figure \ref{fig:rl} plots the \ac{ODE} system dynamics in \eqref{eq:rl_dynamics}. Figure \ref{fig:rl_bound} plots hybrid system dynamics in cyan, starting from the black-circle $X_0$. The $p_6^*$ bound is  indicated in the red circle of radius $\sqrt{-p_6^*} = 0.7564$, in which no considered hybrid system trajectory is contained.

\begin{figure}[h]
    \centering
     \begin{subfigure}[b]{0.45\textwidth}
         \centering
         \includegraphics[width=\textwidth]{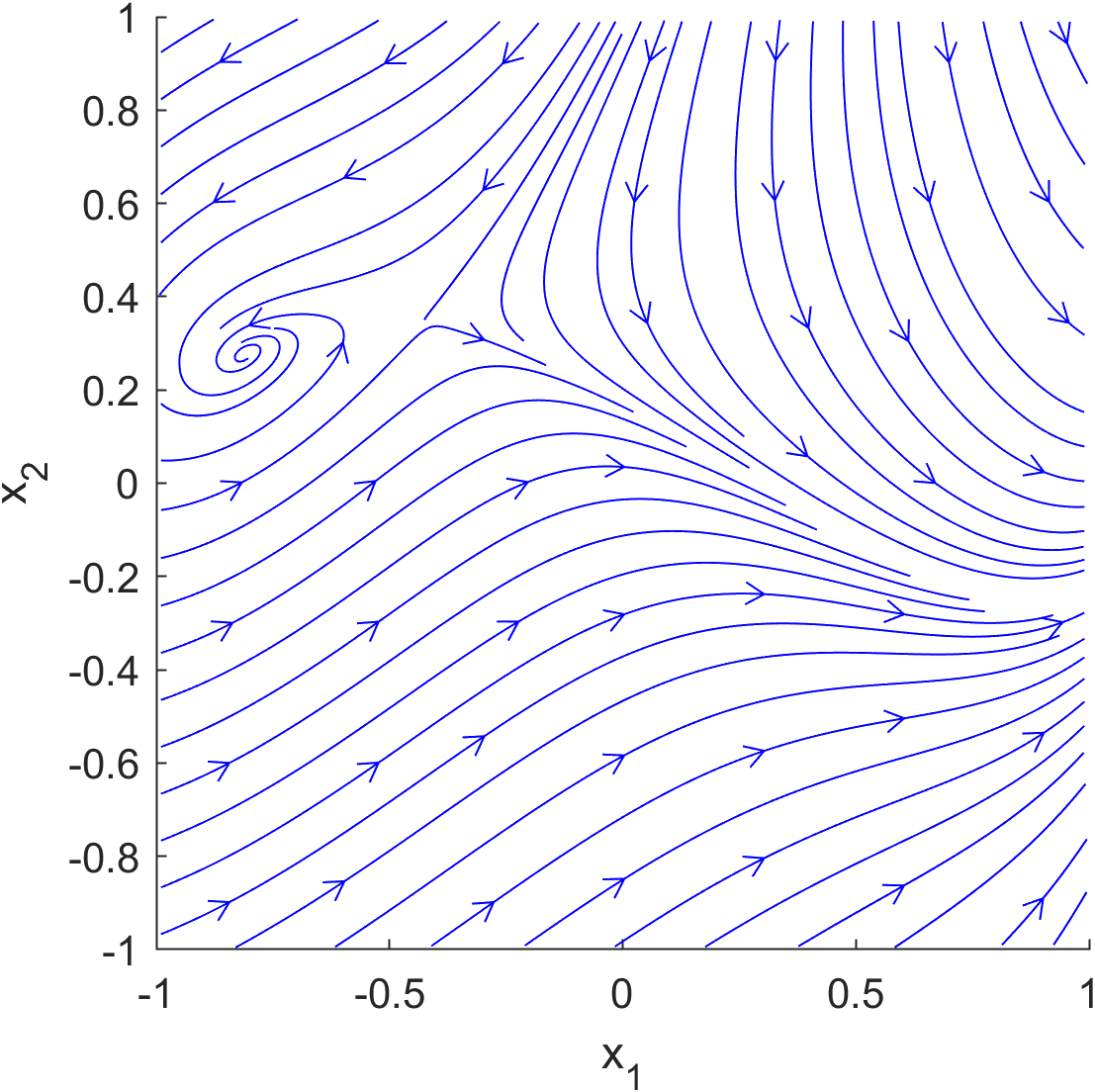}
         \caption{System dynamics in \eqref{eq:rl_dynamics}}
         \label{fig:rl_stream}
     \end{subfigure}
          \hfill
     \begin{subfigure}[b]{0.45\textwidth}
         \centering
         \includegraphics[width=\textwidth]{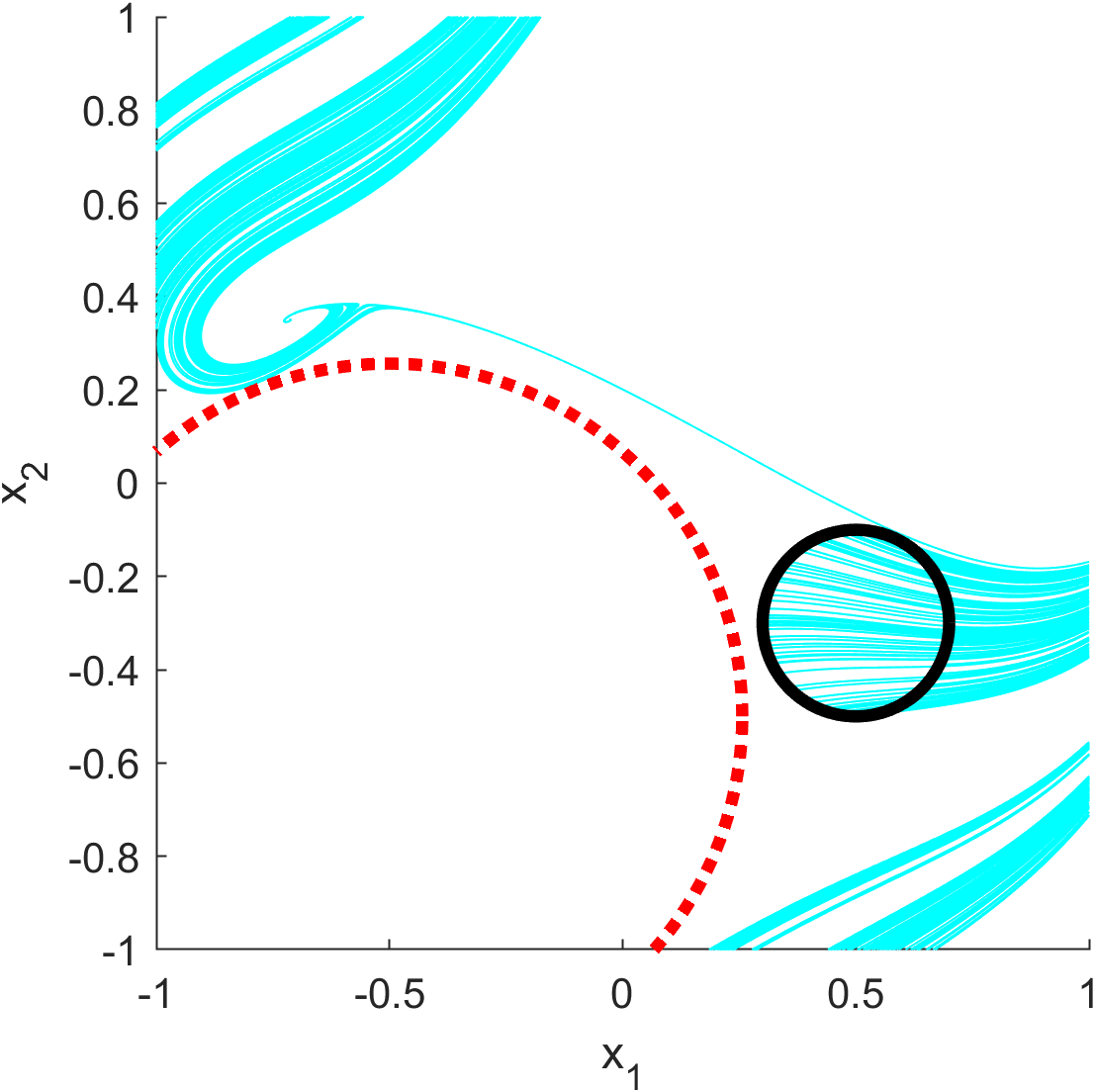}
         \caption{Bound of $p_6^* = -0.5721$}
         \label{fig:rl_bound}
     \end{subfigure}
    \caption{Peak Estimation of Right-Left Wrap Dynamics \eqref{eq:rl_dynamics} and \eqref{eq:rl_transitions}}
    \label{fig:rl}
\end{figure}

\section{Conclusion}
\label{sec:conclusion}

An existing peak estimation framework for \ac{ODE} systems was extended in this paper to hybrid systems. A hierarchy of \acp{SDP} result in a (convergent) decreasing sequence of upper bounds to the true peak value. Extensions to the hybrid peak estimation framework, such as bounding the distance to unsafe sets \cite{miller2021distance} and estimation of systems with uncertainty \cite{miller2021uncertain}, can be accomplished by modifying equations in the \ac{LP}. Future work includes performing peak-minimizing ($L_1$-optimal) control of hybrid systems \cite{molina2022equivalent}, data-driven peak estimation of hybrid systems \cite{miller2021facial_short}, applying numerical techniques to perform peak estimation on more complicated systems (e.g., rigid body dynamics in robotics), and generalizing analysis of deterministic hybrid dynamics to Markov Decision Processes.

\bibliographystyle{IEEEtran}
\bibliography{hybrid.bib}

\begin{thebibliography}{10}
\providecommand{\url}[1]{#1}
\csname url@samestyle\endcsname
\providecommand{\newblock}{\relax}
\providecommand{\bibinfo}[2]{#2}
\providecommand{\BIBentrySTDinterwordspacing}{\spaceskip=0pt\relax}
\providecommand{\BIBentryALTinterwordstretchfactor}{4}
\providecommand{\BIBentryALTinterwordspacing}{\spaceskip=\fontdimen2\font plus
\BIBentryALTinterwordstretchfactor\fontdimen3\font minus
  \fontdimen4\font\relax}
\providecommand{\BIBforeignlanguage}[2]{{%
\expandafter\ifx\csname l@#1\endcsname\relax
\typeout{** WARNING: IEEEtran.bst: No hyphenation pattern has been}%
\typeout{** loaded for the language `#1'. Using the pattern for}%
\typeout{** the default language instead.}%
\else
\language=\csname l@#1\endcsname
\fi
#2}}
\providecommand{\BIBdecl}{\relax}
\BIBdecl

\bibitem{goebel2009hybrid}
R.~Goebel, R.~G. Sanfelice, and A.~R. Teel, ``Hybrid dynamical systems,''
  \emph{IEEE Control Systems Magazine}, vol.~29, no.~2, pp. 28--93, 2009.

\bibitem{collins2005efficient}
S.~Collins, A.~Ruina, R.~Tedrake, and M.~Wisse, ``{Efficient Bipedal Robots
  Based on Passive-Dynamic Walkers},'' \emph{Science}, vol. 307, no. 5712, pp.
  1082--1085, 2005.

\bibitem{zaupa2021resonant}
N.~Zaupa, L.~Martínez-Salamero, C.~Olalla, and L.~Zaccarian, ``Hybrid control
  of self-oscillating resonant converters,'' \emph{IEEE Transactions on Control
  Systems Technology}, vol.~31, no.~2, pp. 881--888, 2023.

\bibitem{naghshtabrizi2006sampled}
P.~Naghshtabrizi, J.~P. Hespanha, and A.~R. Teel, ``On the robust stability and
  stabilization of sampled-data systems: A hybrid system approach,'' in
  \emph{Proceedings of the 45th IEEE Conference on Decision and Control}, 2006,
  pp. 4873--4878.

\bibitem{fischer2012redbloodcell}
\BIBentryALTinterwordspacing
S.~Fischer, P.~Kurbatova, N.~Bessonov, O.~Gandrillon, V.~Volpert, and
  F.~Crauste, ``Modeling erythroblastic islands: Using a hybrid model to assess
  the function of central macrophage,'' \emph{Journal of Theoretical Biology},
  vol. 298, pp. 92--106, 2012. [Online]. Available:
  \url{https://www.sciencedirect.com/science/article/pii/S0022519312000033}
\BIBentrySTDinterwordspacing

\bibitem{zhao2019optimal}
P.~Zhao, S.~Mohan, and R.~Vasudevan, ``{Optimal Control of Polynomial Hybrid
  Systems via Convex Relaxations},'' \emph{IEEE Trans. Automat. Contr.},
  vol.~65, no.~5, pp. 2062--2077, 2019.

\bibitem{lewis1980relaxation}
R.~Lewis and R.~Vinter, ``{Relaxation of Optimal Control Problems to Equivalent
  Convex Programs},'' \emph{Journal of Mathematical Analysis and Applications},
  vol.~74, no.~2, pp. 475--493, 1980.

\bibitem{cho2002linear}
M.~J. Cho and R.~H. Stockbridge, ``{Linear Programming Formulation for Optimal
  Stopping Problems},'' \emph{SIAM Journal on Control and Optimization},
  vol.~40, no.~6, pp. 1965--1982, 2002.

\bibitem{lasserre2009moments}
J.~B. Lasserre, \emph{Moments, Positive Polynomials And Their Applications},
  ser. Imperial College Press Optimization Series.\hskip 1em plus 0.5em minus
  0.4em\relax World Scientific Publishing Company, 2009.

\bibitem{tacchi2022convergence}
M.~Tacchi, ``Convergence of {L}asserre’s hierarchy: the general case,''
  \emph{Optimization Letters}, vol.~16, no.~3, pp. 1015--1033, 2022.

\bibitem{fantuzzi2020bounding}
G.~Fantuzzi and D.~Goluskin, ``{Bounding Extreme Events in Nonlinear Dynamics
  Using Convex Optimization},'' \emph{SIAM J. Appl. Dyn. Syst.}, vol.~19,
  no.~3, pp. 1823--1864, 2020.

\bibitem{lasserre2008nonlinear}
J.~B. Lasserre, D.~Henrion, C.~Prieur, and E.~Tr{\'e}lat, ``{Nonlinear Optimal
  Control via Occupation Measures and LMI-Relaxations},'' \emph{SIAM Journal on
  Control and Optimization}, vol.~47, no.~4, pp. 1643--1666, 2008.

\bibitem{henrion2013convex}
D.~Henrion and M.~Korda, ``{Convex Computation of the Region of Attraction of
  Polynomial Control Systems},'' \emph{IEEE Trans. Automat. Contr.}, vol.~59,
  no.~2, pp. 297--312, 2013.

\bibitem{korda2014convex}
M.~Korda, D.~Henrion, and C.~N. Jones, ``Convex computation of the maximum
  controlled invariant set for polynomial control systems,'' \emph{SIAM Journal
  on Control and Optimization}, vol.~52, no.~5, pp. 2944--2969, 2014.

\bibitem{zhao2019switching}
P.~Zhao and R.~Vasudevan, ``{Nonlinear Hybrid Optimal Control with Switching
  Costs via Occupation Measures and LMI-Relaxations},'' in \emph{2019 American
  Control Conference (ACC)}.\hskip 1em plus 0.5em minus 0.4em\relax IEEE, 2019,
  pp. 4293--4300.

\bibitem{rocca2018bio}
A.~Rocca, M.~Forets, V.~Magron, E.~Fanchon, and T.~Dang, ``Occupation measure
  methods for modelling and analysis of biological hybrid systems,''
  \emph{IFAC-PapersOnLine}, vol.~51, no.~16, pp. 181--186, 2018, 6th IFAC
  Conference on Analysis and Design of Hybrid Systems ADHS 2018.

\bibitem{shia2014reachable}
V.~{Shia}, R.~{Vasudevan}, R.~{Bajcsy}, and R.~{Tedrake}, ``{Convex Computation
  of the Reachable Set for Controlled Polynomial Hybrid Systems},'' in
  \emph{53rd IEEE Conference on Decision and Control}, 2014, pp. 1499--1506.

\bibitem{mohan2016convex}
S.~Mohan and R.~Vasudevan, ``{Convex Computation of the Reachable Set for
  Hybrid Systems with Parametric Uncertainty},'' in \emph{2016 American Control
  Conference (ACC)}, 2016, pp. 5141--5147.

\bibitem{prajna2004safety}
S.~Prajna and A.~Jadbabaie, ``{Safety Verification of Hybrid Systems Using
  Barrier Certificates},'' in \emph{International Workshop on Hybrid Systems:
  Computation and Control}.\hskip 1em plus 0.5em minus 0.4em\relax Springer,
  2004, pp. 477--492.

\bibitem{tao2011introduction}
T.~Tao, \emph{An introduction to measure theory}.\hskip 1em plus 0.5em minus
  0.4em\relax American Mathematical Society Providence, RI, 2011, vol. 126.

\bibitem{hartsfield2013pearls}
N.~Hartsfield and G.~Ringel, \emph{Pearls in Graph Theory: A Comprehensive
  Introduction}.\hskip 1em plus 0.5em minus 0.4em\relax Courier Corporation,
  2013.

\bibitem{ames2006there}
A.~D. Ames, H.~Zheng, R.~D. Gregg, and S.~Sastry, ``{Is there life after Zeno?
  Taking executions past the breaking (Zeno) point},'' in \emph{2006 American
  control conference}.\hskip 1em plus 0.5em minus 0.4em\relax IEEE, 2006, pp.
  6--pp.

\bibitem{miller2020recovery}
J.~{Miller}, D.~{Henrion}, and M.~{Sznaier}, ``{Peak Estimation Recovery and
  Safety Analysis},'' \emph{IEEE Control Systems Letters}, vol.~5, no.~6, pp.
  1982--1987, 2020.

\bibitem{han2019controller}
W.~Han and R.~Tedrake, ``Controller synthesis for discrete-time hybrid
  polynomial systems via occupation measures,'' in \emph{2019 International
  Conference on Robotics and Automation (ICRA)}.\hskip 1em plus 0.5em minus
  0.4em\relax IEEE, 2019, pp. 7675--7682.

\bibitem{putinar1993compact}
M.~Putinar, ``{Positive Polynomials on Compact Semi-algebraic Sets},''
  \emph{Indiana University Mathematics Journal}, vol.~42, no.~3, pp. 969--984,
  1993.

\bibitem{trnovska2005strong}
M.~Trnovsk\'{a}, ``Strong duality conditions in semidefinite programming,''
  \emph{Journal of Electrical Engineering}, vol.~56, no.~12, pp. 1--5, 2005.

\bibitem{miller2021distance}
J.~Miller and M.~Sznaier, ``{Bounding the Distance to Unsafe Sets with Convex
  Optimization},'' 2021, arXiv: 2110.14047.

\bibitem{miller2021uncertain}
J.~Miller, D.~Henrion, M.~Sznaier, and M.~Korda, ``Peak estimation for
  uncertain and switched systems,'' in \emph{2021 60th IEEE Conference on
  Decision and Control (CDC)}.\hskip 1em plus 0.5em minus 0.4em\relax IEEE,
  2021, pp. 3222--3228.

\bibitem{young1942generalized}
L.~C. Young, ``{Generalized Surfaces in the Calculus of Variations},''
  \emph{Annals of mathematics}, vol.~43, pp. 84--103, 1942.

\bibitem{henrion2003gloptipoly}
D.~Henrion and J.-B. Lasserre, ``{GloptiPoly: Global Optimization over
  Polynomials with Matlab and SeDuMi},'' \emph{ACM Transactions on Mathematical
  Software (TOMS)}, vol.~29, no.~2, pp. 165--194, 2003.

\bibitem{lofberg2004yalmip}
J.~{Lofberg}, ``{YALMIP : a toolbox for modeling and optimization in MATLAB},''
  in \emph{ICRA (IEEE Cat. No.04CH37508)}, 2004, pp. 284--289.

\bibitem{mosek92}
\BIBentryALTinterwordspacing
M.~ApS, \emph{The MOSEK optimization toolbox for MATLAB manual. Version 9.2.},
  2020. [Online]. Available:
  \url{https://docs.mosek.com/9.2/toolbox/index.html}
\BIBentrySTDinterwordspacing

\bibitem{molina2022equivalent}
E.~Molina, A.~Rapaport, and H.~Ram{\'\i}rez, ``Equivalent formulations of
  optimal control problems with maximum cost and applications,'' \emph{Journal
  of Optimization Theory and Applications}, vol. 195, no.~3, pp. 953--975,
  2022.

\bibitem{miller2021facial_short}
\BIBentryALTinterwordspacing
J.~Miller and M.~Sznaier, ``{Facial Input Decompositions for Robust Peak
  Estimation under Polyhedral Uncertainty},'' \emph{IFAC-PapersOnLine},
  vol.~55, no.~25, pp. 55--60, 2022. [Online]. Available:
  \url{https://www.sciencedirect.com/science/article/pii/S2405896322015750}
\BIBentrySTDinterwordspacing

\end{thebibliography}

\end{document}